\documentclass[a4paper, 11pt]{amsart} 
\usepackage[
  margin=1.9cm,
  includefoot,
  footskip=30pt,
]{geometry}
\usepackage{amsmath,amssymb,amscd}
\usepackage{amsfonts}
\usepackage[english]{babel}
\usepackage[leqno]{amsmath}
\usepackage{amssymb,amsthm}
\usepackage{mathrsfs} 
\usepackage{amscd}
\usepackage{enumerate}
\usepackage{epsfig}
\usepackage{relsize}
\usepackage{layout}
\usepackage[usenames,dvipsnames]{xcolor}
\usepackage[backref=page]{hyperref}
\usepackage{tikz}
\hypersetup{
 colorlinks,
 citecolor=Green,
 linkcolor=Red,
 urlcolor=Blue}
\usepackage[matrix,arrow,tips,curve]{xy}
\input{xy}
\xyoption{all}
\definecolor{darkgreen}{rgb}{0.0, 0.7, 0.0}
\definecolor{purple}{rgb}{0.5, 0.0, 0.5}
\definecolor{red}{rgb}{0.8, 0.2, 0.0}

\newtheorem{thm}{Theorem}[section]
\newtheorem{bthm}{Theorem}
\newtheorem{bcor}{Corollary}
\newtheorem{lemma}[thm]{Lemma}

\newtheorem{prop}[thm]{Proposition}

\newtheorem{claim}[thm]{Claim}

\numberwithin{equation}{section}
\theoremstyle{definition}
\newtheorem{defi}[thm]{Definition}

\theoremstyle{remark}
\newtheorem{remark}[thm]{Remark}

\newcommand{\Z}{\mathbb{Z}}

\def \P{\mathbb{P}}

\def \F{\mathcal F}

\def \E{\mathcal E}
\def \G{\mathcal G}
\def \H{\mathcal H}

\def\O{\mathcal O}

\def\M0{\mathcal M^0}

\newcommand{\rk}{\operatorname{rank}}

\begin{document}

\title[On the existence of polarized varieties without weakly Ulrich bundles]{On the existence of polarized varieties without weakly Ulrich bundles, ordinary instantons with bounded quantum number and supernatural bundles with bounded roots}

\author[O. Genc, A.F. Lopez]{Ozhan Genc and Angelo Felice Lopez*}

\address{\hskip -.43cm Ozhan Genc, Department of Mathematics, TED University, 06420 Ankara, Turkey.\hfill \null  \hskip 3.1cm email {\tt ozhangenc@gmail.com}}

\address{\hskip -.43cm Angelo Felice Lopez, Dipartimento di Matematica e Fisica, Universit\`a di Roma
Tre, Largo San Leonardo Murialdo 1, 00146, Roma, Italy. e-mail {\tt angelo.lopez@uniroma3.it}}

\thanks{*Partially supported by the GNSAGA group of INdAM.}


\begin{abstract} 
We give an obstruction to the existence of bundles $\E$ on a smooth $n$-dimensional polarized variety such that $\E$ is $a$-regular and $\E^*(K_X+(n+1)H)$ is $b$-regular. In particular, this gives an obstruction to the existence of weakly Ulrich bundles. Using these obstructions and some recent examples of \cite{a1} in dimension $2$, we prove that for any $n \ge 2$ there exist many smooth $n$-dimensional varieties $X$ with a very ample line bundle $H$, such that there is no weakly Ulrich bundle for $(X,H)$, no ordinary instanton bundle with bounded quantum number, no non-ordinary $\mu$-semistable instanton bundle with bounded quantum number and no bundle with supernatural cohomology and bounded roots.
\end{abstract}

\maketitle

\section{Introduction}
Let $k$ be an algebraically closed field. One of the most classical ways of describing the geometry of a variety $X \subset \P^N$ over $k$ of codimension $c$, is to associate to $X$ the hypersurface in $\mathbb G(c-1,N)$ consisting of $(c-1)$-planes meeting $X$ and its corresponding Chow form $[F_X]$. It is indeed a classical fact that $[F_X]$ determines $X$.
Already from this point of view, it is then natural and important to find ways to describe the Chow form.
In their seminal 2003 paper \cite{esw}, Eisenbud, Schreyer and Weyman showed, among other things, that there are particularly simple ways of describing the Chow form of $X$, as soon $X$ carries some special sheaves, namely Ulrich and weakly Ulrich sheaves (see Definition \ref{def}). This originated the question of existence of Ulrich sheaves, already posed in \cite{esw} and subsequently studied by many authors (see for example \cite{b2, c1, c2, cfk1, cfk2, cg, cmr1, cmr2, cohu, cmrpl, esw, f, lo2} and references therein). In a series of recent papers, starting with the breakthrough paper \cite{a1}, and followed by \cite{lo1, a2, a3}, the existence problem for Ulrich bundles (and in fact also for aCM bundles \cite{a2}) was solved in the negative over the complex numbers. The guiding principle of these papers is the fact, coming from an idea of Beauville \cite{b}, described as follows. Let $X \subset \P^N$ be a complex smooth $n$-dimensional variety and let $H$ be a hyperplane divisor. If
$(n+1)H^n < \big(K_X^2-2c_2(X)\big)H^{n-2}$, then there are no Ulrich bundles (\cite{a1, lo1}, indeed no aCM bundles \cite{a2}) for $(X,\O_X(1))$.

Consider now the existence problem for weakly Ulrich bundles. Since the required vanishings are considerably less, it is in general true that there are more weakly Ulrich bundles than Ulrich bundles on a given variety (for a recent evidence, see \cite{bgl}). Despite that, we prove that, over the complex numbers, if
$$(n+73)H^n < \big(K_X^2-2c_2(X)\big)H^{n-2}$$
then there is no weakly Ulrich bundle for $(X,\O_X(1))$.

As a matter of fact, observing that a bundle $\E$ is weakly Ulrich if and only if both $\E$ and $\E^*(K_X+(n+1)H)$ are $1$-regular with respect to $\O_X(1)$ (see Lemma \ref{1-1}), the obstruction to the existence of weakly Ulrich bundles follows by a more general result, of independent interest, that gives an obstruction to the existence of bundles with some regularity properties, as follows.

\begin{bthm}
\label{main1}
Let $X \subseteq \P^N$ be a smooth irreducible variety of dimension $n \ge 2$, let $H \in |\O_X(1)|$ and let $\E$ be a vector bundle on $X$. Suppose that $\E$ is $a$-regular, with respect to $\O_X(1)$, and that $\E^*(K_X+(n+1)H)$ is $b$-regular, for some $a, b \in \Z$. Then $a+b \ge 0$ and
$$H^n \ge \frac{1}{n-2+3(2a+2b+1)^2} \big(K_X^2-2c_2(X)\big)H^{n-2}.$$
\end{bthm}

Then we specialize to the weakly Ulrich case.

\begin{bcor}
\label{nonext}
Let $X \subseteq \P^N$ be a smooth complex irreducible variety of dimension $n \ge 2$ and let $H \in |\O_X(1)|$. If
\begin{equation}
\label{dis}
H^n < \frac{1}{n+73}\big(K_X^2-2c_2(X)\big)H^{n-2}
\end{equation}
then there is no weakly Ulrich bundle for $(X,\O_X(1))$ and for all successive smooth hyperplane sections $(X_i,\O_{X_i}(1)), 2 \le i \le n-1$, where $X_i = X \cap H_1 \cap \ldots \cap H_{n-i}$ with $H_j \in |H|$.
\end{bcor}

Starting from the above obstruction, in the same spirit of \cite{lo1}, we construct the first examples of complex varieties $X \subset \P^N$ that do not carry any weakly Ulrich bundle. 

\begin{bthm}
\label{main2}
For every integer $k \ge 3$, let $Y_k$ be the minimal resolution of the exponent-$k$ Hirzebruch-Kummer cover associated with the Hesse arrangement and let $L_k = 4A_k-E_k$ (see \cite[Thm.~1.2]{a1}).
If $k \ge 8$, there are no weakly Ulrich bundles for $(Y_k,L_k)$. 

Moreover, let $n \ge 3$, let $Z$ be any smooth irreducible $(n-2)$-dimensional variety and let $H_Z$ be any very ample line bundle on $Z$. Let 
$$X=Y_k \times Z \subset \P H^0(H)$$
where $H=L_k \boxtimes H_Z$ and let $X' = X \cap F_s, s \ge 1$ be a smooth hypersurface section, $H'=H_{|X'}$. 
Then for $k \gg 0$ the pairs $(X, H), (X', H')$, and all their successive hyperplane sections until dimension $2$, do not carry a weakly Ulrich bundle. 
\end{bthm}
As observed in \cite{lo1}, the hypersurface sections $X'$ above are not isomorphic to a (non-trivial) product.

We stated the above theorem for weakly Ulrich bundles, but, as a matter of fact, it is far more general. Indeed, as observed in Remark \ref{noab}, the same proof gives that for any given cohomology table of a bundle $\E$ that implies that ${\rm reg}_H(\E)$ and ${\rm reg}_H(\E^*(K_X+(n+1)H))$ are given by the table, we can find pairs $(X,H)$ not having bundles with that cohomology table. 

An instance of this is ordinary and semistable non-ordinary instanton bundles (see Definition \ref{defi}).

\begin{bcor}
\label{inst}
Fix integers $q_0 \ge 0, \gamma_0 \le \gamma_1$. For $k \gg 0$, the pairs $(Y_k,L_k), (X, H), (X', H')$ in Theorem \ref{main2}, and all their successive hyperplane sections until dimension $2$, do not carry neither an ordinary nor a $\mu$-semistable non-ordinary instanton bundle of quantum number $q \le q_0$. Similarly, on the pairs $(Y_k,L_k), (X, H), (X', H')$, there are no bundles with supernatural cohomology and roots of the Hilbert polynomial lying in the interval $[\gamma_0,\gamma_1]$.
\end{bcor}
We cannot say the same for unstable non-ordinary instanton bundles, because one of the regularities can be unbounded, see Proposition \ref{no}.

In the case of Ulrich bundles, perhaps at this point the main question that remains open in characteristic zero is the existence for sufficiently ample polarizations (that is known in the Ulrich case for surfaces by \cite{cohu}). On the other hand, in arbitrary characteristic, a recent result by J. Pine \cite{p} shows that the existence problem has a positive answer for weakly Ulrich bundles on $n$-dimensional varieties for sufficiently ample polarizations. Moreover, in positive characteristic, J. Pine \cite{p} also proves that the existence problem has a positive answer for Ulrich bundles on surfaces and for weakly Ulrich bundles on $n$-dimensional varieties.

\section{Weakly Ulrich bundles}

Following \cite[\S 2]{esw}, we have

\begin{defi}
\label{def}
Let $X \subseteq \P^N$ be a smooth irreducible variety of dimension $n \ge 1$. Let $\E$ be a vector bundle on $X$. We say that $\E$ is a {\it weakly Ulrich vector bundle} for $(X,\O_X(1))$ if the following vanishings hold:
\begin{itemize} 
\item[(i)] $H^i(\E(-t))=0$ for $1 \le i \le n$ and $t \le i-1$.
\item[(ii)] $H^i(\E(-t))=0$ for $0 \le i \le n-1$ and $t \ge i+2$.
\end{itemize} 
\end{defi}

Even though it is probably well-known, we observe, for a lack of reference, that weakly Ulrichness can be characterized by regularity.

\begin{lemma} 
\label{1-1}
Let $X \subseteq \P^N$ be a smooth irreducible variety of dimension $n \ge 1$, let $H \in |\O_X(1)|$ and let $\E$ be a vector bundle on $X$. Then $\E$ is weakly Ulrich for $(X,\O_X(1))$ if and only if both $\E$ and $\E^*(K_X+(n+1)H)$ are $1$-regular with respect to $\O_X(1)$. In particular, if $n \ge 2, Y \in |\O_X(1)|$ is smooth and $\E$ is weakly Ulrich, then $\E_{|Y}$ is weakly Ulrich for $(Y,\O_Y(1))$.
\end{lemma} 
\begin{proof}
Assume that $\E$ is weakly Ulrich and fix $i>0$. Then $H^i(\E(1-i))=0$ by Definition \ref{def}(i), hence $\E$ is $1$-regular. Also, by Serre's duality, $h^i(\E^*(K_X+(n+2-i)H))=h^{n-i}(\E(-(n+2-i)))=0$ by Definition \ref{def}(ii), so that $\E^*(K_X+(n+1)H)$ is $1$-regular. Vice versa, assume that both $\E$ and $\E^*(K_X+(n+1)H)$ are $1$-regular. Then, they are both $m$-regular for every $m \ge 1$ by \cite[Thm. 1.8.5(iii)]{laz}. Now, if $1 \le i \le n$ and $t \le i-1$, set $m=i-t \ge 1$. Then 
$$H^i(\E(-t))=H^i(\E(m-i))=0.$$ 
If $0 \le i \le n-1$ and $t \ge i+2$, set $m=t-i-1 \ge 1$ and $j=n-i>0$. Then, by Serre's duality, 
$$h^i(\E(-t))=h^{n-j}(\E(-i-1-m))=h^j(\E^*(K_X+(n+1+m-j)H))=0$$
so that $\E$ is weakly Ulrich. Now, as $m$-regularity descends to hyperplane sections, we have that
$\E_{|Y}$ is $1$-regular and $\E_{|Y}^*(K_Y+nH_{|Y})=\E^*(K_X+(n+1)H)_{|Y}$ is $1$-regular. Therefore 
$\E_{|Y}$ is weakly Ulrich.
\end{proof}

\section{Instanton bundles}

Following \cite[Def.~1.3]{ac}, we have

\begin{defi}
\label{defi}
Let $X \subseteq \P^N$ be a smooth irreducible variety of dimension $n \ge 2$. Let $\E$ be a non-zero vector bundle on $X$. We say that $\E$ is an {\it instanton bundle (with respect $\O_X(1)$) of defect $\delta \in \{0,1\}$ and quantum number $q \in \Z_{\ge 0}$} if the following conditions hold:
\begin{itemize} 
\item[(i)] $h^0(\E(-t-1))=h^n(\E(\delta-n+t))=0$ for $t \ge 0$.
\item[(ii)] $h^i(\E(-i-t-1))=h^{n-i}(\E(\delta+t-n+i))=0$ for $1 \le i \le n-2$ and $t \ge 0$.
\item[(iii)] $\delta\,h^i(\E(-i))=0$ for $2 \le i \le n-2$.
\item[(iv)] $h^1(\E(-1))=h^{n-1}(\E(\delta-n))=q$.
\item[(v)] $\delta\big(\chi(\E)-(-1)^n\chi(\E(-n))\big)=0$.
\end{itemize} 
$\E$ is called {\it ordinary} if $\delta=0$ and {\it non-ordinary} if $\delta=1$. $\E$ is {\it orientable} if $c_1(\E)=K_X+(n+1-\delta)H$.
\end{defi}

We recall some useful facts about instantons.
 
\begin{remark}
\label{regi}
If $\E$ is an ordinary instanton bundle on $X$ with quantum number $q$, it follows by \cite[Prop.s 6.6 and 5.1]{ac} that both $\E$ and $\E^*(K_X+(n+1)H)$ are $q$-regular (with respect to $\O_X(1)$).
\end{remark}

Hence, when $q$ is bounded, if $\E$ is ordinary, then both regularities of $\E$ and $\E^*(K_X+(n+1)H)$ are bounded.
We cannot say the same for non-ordinary instantons. Indeed, it follows by \cite[Prop.s 6.6 and 5.1]{ac} that if $\E$ is non-ordinary, then ${\rm reg}_H(\E) \le h^1(\E)+1$, that is still independent of $H$, but ${\rm reg}_H(\E^*(K_X+(n+1)H)) \le h^{n-1}(\E(-n))$. The latter can be unbounded, as in the example below.

\begin{prop}
\label{no}
There exist an abelian surface $X \subset \P^N$ and, for every $m \ge 2$, a rank $2$ non-ordinary orientable unstable instanton bundle $\E_m$, of quantum number $2$ and with fixed numerical Chern classes, such that
$$h^1(\E_m(-2)) \ge 3m^2-8.$$
\end{prop}
\begin{proof}
Let $C$ be an elliptic curve with origin $O$, and let $X=C\times C$,
with projections $p_1,p_2$. Let
$$A=p_1^*O, B=p_2^*O \ \hbox{and let } H=3(A+B).$$
Since $\O_C(3O)$ is very ample, it follows that $H$ is very ample and we can consider the embedding $X \subset \P^N=\P H^0(H)$. Moreover,
$$\omega_X\cong \O_X, A^2=B^2=0, A\cdot B=1,H^2=18.$$
For each integer $m\geq2$, consider the automorphism
$$f_m:X\longrightarrow X,\qquad f_m(x,y)=(x+my,y)$$
and set
$$L_m=f_m^*(A-B).$$
Now, the K\"unneth formula gives
$$h^0(\O_X(A-B))=h^2(\O_X(A-B))=0 \ \hbox{and} \ h^1(\O_X(A-B))=1$$
and the same holds for $\O_X(B-A)$. Since pulling back by the automorphism $f_m$ preserves cohomology, we get that
\begin{equation}
\label{cl}
h^0(\pm L_m)=h^2(\pm L_m)=0, \ h^1(\pm L_m)=1.
\end{equation}
We next compute the intersections involving $L_m$.
Since $f_m$ is an automorphism,
$$L_m^2=(A-B)^2=-2.$$
Also, 
$$f_m^*B=B, f_m^*A\cdot A=m^2 \ \hbox{and} \ f_m^*A\cdot B=1.$$
Indeed, restriction of the first component $x+my$ to the curve
$A=\{O\}\times C$ is multiplication by $m$, of degree $m^2$, while restriction to $B=C\times\{O\}$ is the identity. Consequently,
\begin{equation}
\label{hl}
HL_m=3(A+B)(f_m^*A-f_m^*B)=3m^2.
\end{equation}
Define
$$\E_m=\O_X(L_m+H) \oplus \O_X(H-L_m)$$
so that $\E_m$ is rank $2$ with $c_1(\E_m)=2H$. By \eqref{cl},
$$h^0(\E_m(-1))=h^2(\E_m(-1))=0, h^1(\E_m(-1))=2$$
and it follows from \cite[Cor.~6.9]{ac} that $\E_m$ is an orientable non-ordinary instanton bundle of quantum
number $2$. On the other hand
$$h^1(\E_m(-2)) \ge h^1(L_m-H)) \ge -\chi(L_m-H))=-\frac12(L_m-H)^2=3m^2-8.$$
Now 
$$\mu_H(L_m+H)=18+3m^2>18=\mu_H(\E_m)$$
hence $\E_m$ is unstable. Finally, the Chern classes of $\E_m$ are independent of $m$:
$$c_1(\E_m)=2H, \deg c_2(\E_m)=(H+L_m)(H-L_m)=20.$$
\end{proof}

We now give an obstruction to the existence of non-ordinary $\mu$-semistable instanton bundles with given quantum number. We observe that the surface in Proposition \ref{no} does not satisfy the obstruction.

\begin{lemma}
\label{noinst}
Let $X \subseteq \P^N$ be a smooth irreducible variety of dimension $n \ge 2$ and let $H \in |\O_X(1)|$. Let $\E$ be a rank $r$ non-ordinary $\mu_H$-semistable instanton bundle on $X$ with quantum number $q$. Then
\begin{equation}
\label{dis1}
(n-2)H^n \ge \big(K_X^2-2c_2(X)\big)H^{n-2}-\frac{48q}{r}.
\end{equation}
\end{lemma}
\begin{proof}
Set $d=H^n$. We start by proving something a little stronger than \eqref{dis1} for $n=2$.

\renewcommand{\proofname}{Proof}
\begin{claim}
\label{primo}
Let $S \subseteq \P^N$ be a smooth irreducible surface. Let $\F$ be a rank $f$ non-ordinary instanton bundle on $S$ with quantum number $p$. If $\Delta(\F) \ge 0$, then
\begin{equation}
\label{dis3}
K_S^2-2c_2(S) \le \frac{24p}{f}.
\end{equation}
\end{claim}
\begin{proof} 
Set $s=\deg S$. We first show that
\begin{equation}
\label{chiis1}
\chi(\F)=\frac{fs}{2}-p.
\end{equation}
In fact, by Riemann-Roch, the Hilbert polynomial of $\F$ is
$$P(\ell):=\chi(\F(\ell))=\frac{fs}{2}\ell^2+\big(c_1(\F)H-\frac{f}{2}K_SH\big)\ell+\chi(\F).$$
For a quadratic polynomial $P(t)=At^2+Bt+C$, its second difference is
$$P(0)-2P(-1)+P(-2)=C-2(A-B+C)+(4A-2B+C)=2A$$
consequently
\begin{equation}
\label{diff1}
\chi(\F)-2\chi(\F(-1))+\chi(\F(-2))=fs.
\end{equation}
Now, the non-ordinary instanton conditions give $H^0(\F(-1))=H^2(\F(-1))=0$ and $h^1(\F(-1))=p$, hence $\chi(\F(-1))=-p$. They also impose $\chi(\F)=\chi(\F(-2))$, and therefore, replacing in \eqref{diff1}, we get
$$2\chi(\F)+2p=fs$$
hence \eqref{chiis1}. Set now
$$\alpha = c_1(\F)-\frac{f}{2}(K_S+2H)$$
so that $\alpha H=0$ by \cite[Thm.~1.6]{ac}, hence $\alpha^2 \le 0$ by the Hodge index theorem. Now, Riemann-Roch gives, using \eqref{chiis1},
$$\begin{aligned}
0 \le \Delta(\F) & =\big(c_1(\F)-\frac{f}{2}K_S\big)^2+\frac{f^2}{4}\big(8\chi(\O_S)-K_S^2\big)-2f\chi(\F)= \\
& \alpha^2+\frac{f^2}{4}\big(8\chi(\O_S)-K_S^2\big)+2fp \le \frac{f^2}{4}\big(8\chi(\O_S)-K_S^2\big)+2fp
\end{aligned}$$
hence
\begin{equation}
\label{dis2}
K_S^2-2c_2(S)=3\big(K_S^2-8\chi(\O_S)\big) \le \frac{24p}{f}
\end{equation}
hence \eqref{dis3} holds.
\end{proof} 
\renewcommand{\proofname}{Proof}

We continue the proof of Lemma \ref{noinst}. 

Since $\E$ is $\mu_H$-semistable, it follows by Bogomolov's inequality that $\Delta(\E)H^{n-2}  \ge 0$. 

Now, when $n=2$, applying Claim \ref{primo} to $\E$, we have \eqref{dis3}, hence also \eqref{dis1}.

Suppose next that $n \ge 3$ and consider $\E'=\E_{|S}$, where $S=X \cap H_1 \cap \ldots \cap H_{n-2}$, with $H_j \in |H|, 1 \le j \le n-2$ general. It follows by \cite[Cor.~4.5]{ac} that $\E'$ is a rank $r$ non-ordinary instanton bundle on $S$ with quantum number $p=2q$. Moreover
$$\Delta(\E')=2rc_2(\E_{|S})-(r-1)c_1(\E_{|S})^2=\big(2rc_2(\E)-(r-1)c_1(\E)^2\big)H^{n-2}=\Delta(\E)H^{n-2} \ge 0.$$
Therefore Claim \ref{primo} applies to $\E'$, and we deduce that
\begin{equation}
\label{s2}
K_S^2-2c_2(S) \le \frac{48q}{r}.
\end{equation}
On the other hand, $K_S=(K_X+(n-2)H)_{|S}$, hence
$$K_S^2=K_X^2H^{n-2}+(n-2)^2d+2(n-2)K_XH^{n-1}$$
and (see for example \cite[Proof of Lemma 3.2]{lr2}),
$$c_2(S)=c_2(X) H^{n-2}+(n-2) K_X H^{n-1} + \binom{n-1}{2}d$$
so that \eqref{s2} becomes
$$K_X^2H^{n-2}+(n-2)^2d-2c_2(X) H^{n-2}-2\binom{n-1}{2}d \le \frac{48q}{r}$$
or, equivalently,
$$(n-2)H^n=\left(-(n-2)^2+2\binom{n-1}{2}\right)d \ge \big(K_X^2-2c_2(X)\big)H^{n-2}-\frac{48q}{r}.$$
\end{proof}

\section{An obstruction to the existence of bundles with some regularity properties}

In this section we prove Theorem \ref{main1}.

\renewcommand{\proofname}{Proof of Theorem \ref{main1}}
\begin{proof}
We first show that the theorem holds when $n=2$. Hence we assume that $\E$ is $a$-regular and $\E^*(K_X+3H)$ is $b$-regular. Set
$$r=\rk (\E),\qquad d=H^2,\qquad k=K_X H, \qquad \mu=\mu_H(\E)$$
and, as usual,
$$\Delta(\E)=2rc_2(\E)-(r-1)c_1(\E)^2.$$
Consider a finite linear projection $\pi: X \to \P^2$. We have $\pi^*\O_{\P^2}(1)=\O_X(1)$ and $\pi$ is flat by \cite[Ex.~III.9.3(a)]{h}, since it is finite and surjective. Therefore 
$$\G=\pi_*\E$$ 
is a vector bundle on $\P^2$ by Grauert's criterion \cite[Cor.~III.12.9]{h}. Moreover, finiteness and the projection formula give
\begin{equation}
\label{leray}
H^i(\G(t)) \cong H^i(\E(t)) \ \hbox{for all} \ t \in \Z, i \ge 0. 
\end{equation}
Since $\E$ is $a$-regular, we have that $H^1(\E(a-1))=H^2(\E(a-2))=0$, hence, using \eqref{leray}, $H^1(\G(a-1))=H^2(\G(a-2))=0$, so that $\G(a)$ is $0$-regular and therefore globally generated. Similarly, $h^1(\G^*(b-1))=h^1(\G(-2-b))=h^1(\E(-2-b))=h^1(\E^*(K_X+(b+2)H)=0$ and $h^2(\G^*(b-2))=h^0(\G(-b-1))=h^0(\E(-b-1))=h^2(\E^*(K_X+(b+1)H)=0$, hence $\G^*(b)$ is also $0$-regular, hence globally generated. Then we have a surjection $\O_{\P^2}^{\oplus \beta} \to \G^*(b)$ and therefore an inclusion
\begin{equation}
\label{g1}
\G \subseteq \O_{\P^2}(b)^{\oplus \beta}.
\end{equation}
Consider the Harder-Narashiman filtration
$$0=\E_0 \subset \E_1 \subset \ldots \subset \E_m=\E$$
with $\mu_H$-semistable quotients $\F_i = \E_i / \E_{i-1}$, for $1 \le i \le m$. 

We will now divide the proof into several claims.
\renewcommand{\proofname}{Proof}
\begin{claim} 
\label{slopes}
Let $h=\O_{\P^2}(1)$. Then $\mu_h(\pi_* \F_1) \le b$ and $\mu_h(\pi_* \F_m) \ge -a$.
\end{claim}
\begin{proof} 
Since $\F_1 = \E_1 \subset \E$, we have an exact sequence
$$0 \to \F_1 \to \E \to \E/\F_1 \to 0$$
and, being $R^1 \pi_* \F_1=0$ because $\pi$ is finite, we get an exact sequence
$$0 \to \pi_* \F_1 \to \G \to \pi_*(\E/\F_1) \to 0.$$
Combining with \eqref{g1}, we find an inclusion
$$\pi_* \F_1 \subseteq \O_{\P^2}(b)^{\oplus \beta}$$
and, since $\O_{\P^2}(b)^{\oplus \beta}$ is $\mu_h$-semistable with $\mu_h(\O_{\P^2}(b)^{\oplus \beta})=b$, we have $\mu_h(\pi_* \F_1) \le b$. 

The exact sequence
$$0 \to \E_{m-1} \to \E \to \F_m \to 0$$
gives rise, as above, to an exact sequence
$$0 \to \pi_* \E_{m-1} \to \G \to \pi_*\F_m \to 0$$
showing that $(\pi_*\F_m)(a)$ is globally generated, because $\G(a)$ is. Thus,
$$0 \le \mu_h((\pi_* \F_m)(a))=\mu_h(\pi_* \F_m)+a$$
and therefore $\mu_h(\pi_* \F_m) \ge -a$.
This proves Claim \ref{slopes}.
\end{proof} 
\renewcommand{\proofname}{Proof}

We continue the proof of Theorem \ref{main1}.

\begin{claim} 
\label{slopes2}
We have $\mu \le \mu_{max}(\E) \le \frac{k+3d+2bd}{2}$ and $\mu \ge \mu_{min}(\E) \ge \frac{k+3d-2ad}{2}$.
\end{claim}
\begin{proof} 
For every torsion-free sheaf $\F$ of rank $f$ on $X$, we have, as in \eqref{leray}, $H^i(\F(\ell)) \cong H^i((\pi_* \F)(\ell))$
for all $\ell \in \Z, i \ge 0$ and therefore $\F$ and $\pi_* \F$ have the same Hilbert polynomial. Now, there are two constants, $u, v$ such that
$$\chi(\F(\ell))=\frac{fd}{2}\ell^2+(c_1(\F)H-\frac{fk}{2})\ell+u$$
and
$$\chi((\pi_* \F)(\ell))=\frac{\rk(\pi_* \F)}{2}\ell^2+\big(c_1(\pi_* \F)h+\frac{3\rk(\pi_* \F)}{2}\big)\ell+v$$
and equality of Hilbert polynomials yields
$$\rk(\pi_* \F)=fd, \qquad c_1(\pi_* \F)h=c_1(\F)H-\frac{f}{2}(k+3d)$$
hence
$$\mu_h(\pi_* \F)=\frac{\mu_H(\F)}{d}-\frac{k+3d}{2d}.$$
It follows from Claim \ref{slopes} that
$$\mu \le \mu_{max}(\E)=\mu_H(\F_1)=d\mu_h(\pi_* \F_1)+\frac{k+3d}{2} \le \frac{k+3d+2bd}{2}$$
and that
$$\mu \ge \mu_{min}(\E)=\mu_H(\F_m)=d\mu_h(\pi_* \F_m)+\frac{k+3d}{2} \ge \frac{k+3d-2ad}{2}.$$
This proves Claim \ref{slopes2}.
\end{proof} 
\renewcommand{\proofname}{Proof}

\begin{claim} 
\label{inp2}
$a+b \ge 0$.
\end{claim}
\begin{proof} 
Set $\H=\E^*(K_S+(4-a)H)$ and $\F=(\pi_* \H)^*$. As in \eqref{leray}, we have that $H^j(\H(\ell)) \cong H^j(\F^*(\ell))$
for all $\ell \in \Z, j \ge 0$. We first prove that $\F$ is is $1$-regular. Indeed, for $i>0$ we have, by Serre's duality
$$h^i(\F(1-i))=h^{2-i}(\F^*(i-4))=h^{2-i}(\H(i-4))=h^{2-i}(\E^*(K_S+(i-a)H))=h^i(\E(a-i))=0$$
since $\E$ is $a$-regular. Now, let us see that $\F^*$ is $a+b-1$-regular. In fact, we have
$$h^i(\F^*(a+b-1-i))=h^i(\H(a+b-1-i))=h^i(\E^*(K_S+(b+3-i)H))=0$$
since $\E^*(K_S+3H)$ is $b$-regular.
Considering the Harder-Narashiman filtration of $\F$, with the same proof as Claim \ref{slopes}, we have that
$$-1 \le \mu_{min}(\F) \le \mu_{max}(\F) \le a+b-1$$
hence $a+b \ge 0$. This proves Claim \ref{inp2}.
\end{proof} 
\renewcommand{\proofname}{Proof}

To continue with the proof of Theorem \ref{main1}, we now apply the Langer-Bogomolov inequality \cite[Prop.~1.3.2]{la}
$$d \Delta(\E)+r^2\big(\mu_{\max}(\E)-\mu \big) \big(\mu-\mu_{\min}(\E)\big) \ge 0$$
to get, by Claim \ref{slopes2},
\begin{equation}
\label{bl}
4d \Delta(\E) \ge -4r^2\big(\mu_{\max}(\E)-\mu \big) \big(\mu-\mu_{\min}(\E)\big) \ge -r^2\big(k+3d+2bd-2\mu\big)\big(2\mu-k-3d+2ad\big).
\end{equation}
Since $H^1(\E(a-1))=0$, we have that $\chi(\E(a-1)=h^0(\E(a-1))+h^2(\E(a-1)) \ge 0$, hence Riemann-Roch gives 
\begin{equation}
\label{chi}
\begin{aligned}
0 \le r\chi(\E(a-1))= & r^2\chi(\O_X)-\frac{r^2K_X^2}{8}+\frac{1}{2}\left(c_1(\E(a-1))-\frac{rK_X}{2}\right)^2-\frac{\Delta(\E(a-1))}{2} = \\ 
& r^2\chi(\O_X)-\frac{r^2K_X^2}{8}+\frac{1}{2}\left(c_1(\E(a-1))-\frac{rK_X}{2}\right)^2-\frac{\Delta(\E)}{2}
\end{aligned}
\end{equation}
By the Hodge index theorem we have
$$\left(c_1(\E(a-1))-\frac{rK_X}{2}\right)^2 \le \frac{\left((c_1(\E(a-1))-\frac{rK_X}{2})H\right)^2}{d} = \frac{r^2\big(\mu+ad-d-\frac{k}{2}\big)^2}{d}$$
and we obtain from \eqref{chi} and \eqref{bl} that 
$$0 \le r^2\chi(\O_X)-\frac{r^2K_X^2}{8}+ \frac{r^2\big(\mu+ad-d-\frac{k}{2}\big)^2}{2d}+\frac{r^2\big(k+3d+2bd-2\mu\big)\big(2\mu-k-3d+2ad\big)}{8d}$$
and therefore
\begin{equation}
\label{chi2}
\begin{aligned}
0 \le & 8\chi(\O_X)-K_X^2+\frac{4\big(\mu+ad-d-\frac{k}{2}\big)^2+\big(k+3d+2bd-2\mu\big)\big(2\mu-k-3d+2ad\big)}{d}= \\
& 8\chi(\O_X)-K_X^2+2(a+b+1)(2 \mu-k)+d(4a^2+ 4 a b-2a-6b-5).
\end{aligned}
\end{equation}
By Claim \ref{slopes2} we have
\begin{equation}
\label{mun}
2\mu-k \le (3+2b)d.
\end{equation}
Then \eqref{chi2} and \eqref{mun} give
$$0 \le 8\chi(\O_X)-K_X^2+2(a+b+1)(3+2b)d+d(4a^2+4ab-2a-6b-5)=8\chi(\O_X)-K_X^2+(2a+2b+1)^2d$$
and then
\begin{equation}
\label{h1}
H^2=d \ge \frac{1}{(2a+2b+1)^2}\big(K_X^2-8\chi(\O_X)\big)=\frac{1}{3(2a+2b+1)^2}\big(K_X^2-2c_2(X)\big).
\end{equation}
This proves the theorem when $n=2$.

Now we deal with dimension $n \ge 3$. Let $S=X \cap H_1 \cap \ldots \cap H_{n-2}$ be a smooth surface section, with $H_j \in |H|, 1 \le j \le n-2$. As $m$-regularity descends to hyperplane sections, we have that
$\E_{|S}$ is $a$-regular and $\E_{|S}^*(K_S+3H_{|S})=\E^*(K_X+(n+1)H)_{|S}$ is $b$-regular. Set
$$c = 3(2a+2b+1)^2.$$
Therefore, by the case $n=2$ and observing that $d=H^n=H_{|S}^2$, we have that 
\begin{equation}
\label{sus}
cH^n=cd \ge K_S^2-2c_2(S).
\end{equation}
On the other hand, $K_S=(K_X+(n-2)H)_{|S}$, hence
$$K_S^2=K_X^2H^{n-2}+(n-2)^2d+2(n-2)K_XH^{n-1}$$
and (see for example \cite[Proof of Lemma 3.2]{lr2}),
$$c_2(S)=c_2(X) H^{n-2}+(n-2) K_X H^{n-1} + \binom{n-1}{2}d$$
so that \eqref{sus} becomes
$$cd \ge K_X^2H^{n-2}+(n-2)^2d-2c_2(X) H^{n-2}-2 \binom{n-1}{2}d$$
or, equivalently,
$$(n-2+c)H^n=\left(c-(n-2)^2+2\binom{n-1}{2}\right)d \ge \big(K_X^2-2c_2(X)\big)H^{n-2}$$
and the theorem is proved. 
\end{proof}
\renewcommand{\proofname}{Proof}

\section{Proof of Corollary \ref{nonext}, Theorem \ref{main2} and Corollary \ref{inst}}

We start with Corollary \ref{nonext}.

\renewcommand{\proofname}{Proof of Corollary \ref{nonext}}
\begin{proof}
It follows by Lemma \ref{1-1}, that a bundle $\E$ is weakly Ulrich for $(X,\O_X(1))$ if and only if both $\E$ and $\E^*(K_X+(n+1)H)$ are $1$-regular with respect to $\O_X(1)$. Hence, if there exists a weakly Ulrich bundle for $(X,\O_X(1))$, Theorem \ref{main1} gives that 
$$H^n \ge \frac{1}{n+73}\big(K_X^2-2c_2(X)\big)H^{n-2}$$
contradicting \eqref{dis}. This proves the first assertion of the Corollary.

Finally, exactly as in \cite[Proof of Lemma 3.1]{lo1}, \eqref{dis} is equivalent to the corresponding one on a smooth hyperplane section of $X$. Therefore, there are no weakly Ulrich bundles on all successive smooth hyperplane sections $(X_i,H_{|X_i}), 2 \le i \le n-1$.
\end{proof}
\renewcommand{\proofname}{Proof}

Next, we prove Theorem \ref{main2}.

\renewcommand{\proofname}{Proof of Theorem \ref{main2}}
\begin{proof}
Set $Y=Y_k$ and $L=L_k$. By \cite[Thm.~1.2]{a1} we have that 
$$L^2=7k^9 \ \hbox{and} \ K_Y^2-2c_2(Y)=3\sigma(Y)=3(3k^2-11)k^9>0.$$
Therefore, \eqref{dis} for $(Y_k,L_k)$ is equivalent to
$$\frac{L^2}{\sigma(Y)}=\frac{7}{3k^2-11} <\frac{1}{25}$$
which holds for $k \ge 8$. Thus, there are no weakly Ulrich bundles for $(Y_k,L_k), k \ge 8$ by Corollary \ref{nonext}. 

As for $X=Y \times Z$, exactly with the same computation in \cite[Proof of Thm.~1]{lo1} we see that, when $n \ge 4$, \eqref{dis} is equivalent to
\begin{equation}
\label{f1}
\left((n+73)\binom{n}{2}H_Z^{n-2}-\binom{n-2}{2}K_Z^2H_Z^{n-4}+2\binom{n-2}{2}c_2(Z)H_Z^{n-4}\right)L^2 < H_Z^{n-2}\big(K_Y^2-2c_2(Y)\big).
\end{equation}
Set
$$g(n,Z,H_Z)=(n+73)\binom{n}{2}H_Z^{n-2}-\binom{n-2}{2}K_Z^2H_Z^{n-4}+2\binom{n-2}{2}c_2(Z)H_Z^{n-4}$$
so that \eqref{f1} becomes
\begin{equation}
\label{f2}
g(n,Z,H_Z)L^2 < H_Z^{n-2}\big(K_Y^2-2c_2(Y)\big).
\end{equation}
Therefore, if $g(n,Z,H_Z) \le 0$, then certainly \eqref{f2} holds. On the other hand, if $g(n,Z,H_Z)>0$, \eqref{f2} is equivalent to
$$\frac{L^2}{\sigma(Y)}=\frac{7}{3k^2-11} <\frac{3H_Z^{n-2}}{g(n,Z,H_Z)}$$
that certainly holds for $k \gg 0$.

Similarly, if $n=3$, as in \cite[Proof of Thm.~1]{lo1}, we see that \eqref{dis} is holds if and only if
$$L^2 < \frac{\sigma(Y)}{76}$$
that is to
$$\frac{L^2}{\sigma(Y)}=\frac{7}{3k^2-11} <\frac{1}{76}$$
that certainly holds for $k \ge 14$.

Therefore, in all cases, \eqref{dis} is satisfied and we conclude by Corollary \ref{nonext} that $(X, H)$ and all its successive hyperplane sections until dimension $2$, do not carry a weakly Ulrich vector bundle.

As for $(X', H')$, exactly with the same computation in \cite[Proof of Thm.~1]{lo1}, we see that \eqref{dis} holds if and only if
$$(n+72+s^2)H^n<\big(K_X^2-2c_2(X)\big)H^{n-2}.$$
Again, as in \cite[Proof of Thm.~1]{lo1}, the latter holds for $k \gg 0$ and then $(X', H')$ does not carry a weakly Ulrich vector bundle by Corollary \ref{nonext}.
\end{proof}
\renewcommand{\proofname}{Proof}

\begin{remark}
\label{noab}
A simple application of Theorem \ref{main1} actually shows that, for $k \gg a, b$, the pairs $(X, H)$, $(X', H')$, and all their successive hyperplane sections until dimension $2$, do not carry any bundle $\E$ such that $\E$ is $a$-regular and $\E^*(K_X+(n+1)H)$ is $b$-regular. 
\end{remark}

We now prove Corollary \ref{inst}.

\renewcommand{\proofname}{Proof of Corollary \ref{inst}}
\begin{proof}
Given a smooth projective variety $X \subset \P^N$ and an ordinary instanton bundle $\E$ on $X$ with quantum number $q \le q_0$, it follows by Remark \ref{regi} that both $\E$ and $\E^*(K_X+(n+1)H)$ are $q_0$-regular. Now the corollary follows, in the ordinary  instanton case, by Remark \ref{noab}. Assume instead that $\E$ is a bundle with supernatural cohomology and roots of the Hilbert polynomial lying in the interval $[\gamma_0,\gamma_1]$, that is the roots $z_1 > \ldots > z_n$ satisfy $z_1 \le \gamma_1$ and $z_n \ge \gamma_0$. It follows from \cite[Thm. 6.4]{es} that $\E$ is $(z_1+1)$-regular, hence also $(\gamma_1+1)$-regular and $\E^*(K_X+(n+1)H)$ is $(-z_n-n)$-regular, hence also $(-\gamma_0-n)$-regular. Again, the corollary follows, in the supernatural case, by Remark \ref{noab}. 

Finally, assume that $\E$ is a rank $r$ non-ordinary $\mu$-semistable instanton bundle on $X \subset \P^N$ with quantum number $q \le q_0$. Then
$$(n-2)H^n \ge \big(K_X^2-2c_2(X)\big)H^{n-2}-\frac{48q}{r}$$
by Lemma \ref{noinst}. Therefore we will be done if we show that, in all three cases $(Y_k,L_k), (X, H), (X', H')$, the inequality
$$(n-2)H^n < \big(K_X^2-2c_2(X)\big)H^{n-2}-48q_0$$
holds. But the latter follows exactly as in the proof of Theorem \ref{main2}, by replacing $n+73$ with $n-2$. This completes the proof of Corollary \ref{inst}. 
\end{proof}
\renewcommand{\proofname}{Proof}

\end{document}